\documentclass[onefignum,onetabnum]{siamart190516}
\usepackage{epsfig}
\title{A note on indefinite matrix splitting and preconditioning}
\author{Andy Wathen\thanks{Mathematical Institute, Oxford University, UK (wathen@maths.ox.ac.uk)
and Computational Mathematics Group, Rutherford Appleton Laboratory, UK (andrew.wathen@stfc.ac.uk)}}
\ifpdf
\hypersetup{
  pdftitle={Indefinite matrix splitting and preconditioning},
  pdfauthor={A. J. Wathen}
}
\fi

\usepackage{graphicx}
\usepackage{amsmath}
\usepackage{amsfonts}
\usepackage{MnSymbol}

\newtheorem*{thm*}{Theorem}[section]
\newtheorem*{lem*}[thm*]{Lemma}
\newtheorem*{prop*}[thm*]{Proposition}

\newtheorem*{exmpl*}[thm*]{{\it Example}}
\newtheorem*{rem*}[thm*]{{\it Remark}}
\newtheorem*{cor*}[thm*]{Corollary}

\def\smallskip{\vskip 5pt plus1pt minus1pt}
\def\medskip{\vskip 9pt plus2pt minus2pt}
\def\bigskip{\vskip 18pt plus4pt minus4pt}

\def\half{{\frac{1}{2}}}

\def\R{\hbox{{\msbm \char "52}}}

\def\R{{\mathbb{R}}}

\def\sqr#1#2{{\vcenter{\vbox{\hrule height.#2pt
        \hbox{\vrule width.#2pt height#1pt \kern#1pt \vrule width.#2pt}
             \hrule height.#2pt}}}}

\def\minres{{\large {\sc minres}}}
\def\gmres{{\large {\sc gmres}}}

\begin{document}

\maketitle
\begin{center}{Dedicated to Daniel Szyld on his 70th birthday}\end{center}
\medskip

% REQUIRED
\begin{abstract}
  The solution of systems of linear(ized) equations lies at the heart of
  many problems in Scientific Computing. In particular for systems of large
  dimension, iterative methods are a primary approach.
  Stationary iterative methods are generally based on a matrix splitting, whereas
  for polynomial iterative methods such as Krylov subspace iteration, the splitting
  matrix is the preconditioner. The smoother in a multigrid method is generally a
  stationary or polynomial iteration.

  Here we consider real symmetric indefinite and complex Hermitian indefinite coefficient
  matrices and prove that no splitting matrix can lead to a contractive stationary
  iteration unless the inertia is exactly preserved. This has consequences for
  preconditioning for indefinite systems and smoothing for multigrid as we further describe.
  
\end{abstract}

% REQUIRED

% REQUIRED

\section*{Introduction and Main Result}

%\lipsum[2-3]
An elementary approach for the iterative solution of an invertible linear system of equations
$$A x = b$$
is to split $$A = M-N$$ and from some starting vector $x_0$ to iterate
\begin{equation}{\rm{solve}}\quad M x_{k+1} = N x_{k} + b, \quad k= 0,1,\ldots\  .\label{statit}
  \end{equation}
$M$ is here called the splitting matrix or the preconditioner and it is required to be invertible.
For any $x_0$ the sequence $\{x_k\}$ converges to the solution if and only if all eigenvalues
of $M^{-1} N = I-M^{-1} A$ lie strictly inside the unit disc; in such a situation the
iteration is contractive. Equivalently, all eigenvalues of $M^{-1} A$ must lie in $B(1,1)$, the
open unit ball centred at $1$. In particular, if ever $M^{-1}A$ has an eigenvalue with negative
real part then (\ref{statit}) certainly can not be contractive. For the situation when $A$ is
singular but the equations are consistent, see \cite{Szyld}. 

More recent use of preconditioning has utilised polynomial iterative methods, either with
fixed polynomial sequences as in Chebyshev (semi-)iteration or implicitly defined polynomials
as in Krylov subspace iteration (for both, see for example \cite{Saad},\cite{Simoncini_Szyld}).
\color{black}

Krylov subspace methods are typically applied to the preconditioned system $M^{-1} A x = M^{-1} b$
so that it is polynomials in $M^{-1} A$ that are relevant in the context of convergence; the
polynomials $p(\cdot)$ \color{black}must satisfy the consistency
condition $p(0)=1$.
\color{black}
%(or in some cases the constraint point is not at $0$ but at $1$ so the consistency
%condition is $p(1)=1$, see
%for example \cite{Eiermann_Niethammer_Varga}).
\color{black}
Convergence of a Krylov subspace iteration
will be rapid if $\max{|p(\lambda)|}$ reduces significantly as the degree of the polynomials
$p$ is increased. Here the maximum is over all eigenvalues $\lambda$ of $M^{-1}A$.
Rapid convergence is typically associated with clustering of these eigenvalues away from
the 
\color{black}
origin \cite[Chapter 3]{Greenbaum}.
If all eigenvalues are real and lie only to one side of the origin, then
a Conjugate Gradient Krylov subspace method can be utilised.
%or if a fixed polynomial sequence is required, then a
%Chebyshev polynomial method based on a single interval that contains the eigenvalues can be employed.
If all eigenvalues of $M^{-1}A$ are
real but lie on both sides of the origin, then
a less efficient minimum residual method such as \minres\  is typically used and convergence can be 
less rapid (see for example \cite{ESW}).
%It is generally more difficult to identify a convenient fixed polynomial sequence when there are
%real eigenvalues on
%either side of the constraint point.
%In particular, if $M^{-1} A$ has both positive and negative eigenvalues, then using a Chebyshev iteration
%based on any single interval can not be convergent.
Whenever there are non-real eigenvalues
%(with non-vanishing imaginary parts)
then the more general (and more expensive) \gmres\ iterative
method \cite{Saad_Schultz86} is ubiquitously
employed \color{black} as a Krylov subspace iterative solver.
%If a fixed polynomial sequence is required, then
%typically complex Chebyshev polynomials must be employed (see for example \cite[Section 6.11]{Saad}.
%In this situation it is more difficult to ensure convergence if the eigenvalues surround the
%constraint point.

In terms of methods with fixed polynomial sequences, usually it is polynomials in $I-M^{-1}A$ that
are employed so that the consisitency condition is $p(1)=1$. The classical Chebyshev semi-iterative
method based on the Chebyshev polynomials on a single interval (see for example \cite[Chapter 5]{Varga})
can be used when all eigenvalues are real and lie to one side of unity. It is more difficult to identify
polynomials which lead to an efficient iterative method when all eigenvalues are real, but lie to both
sides of $1$. For the general situation with complex eigenvalues, see for example
\cite{Eiermann_Niethammer_Varga}.

It is thus possible to have a convergent polynomial iterative method for all eigenvalue distributions.
Nevertheless, when eigenvalues are all real, convergence is usually slower when $M^{-1} A$ has eigenvalues
with different signs rather than just of one sign.

\color{black}

If $A$ is real and symmetric, it's {\em{inertia}} (the triple of number of positive,
zero and negative eigenvalues) is $(p,z,n)$ for some integers $p,z,n\geq 0$.
$z=0$ for any invertible matrix. A matrix with inertia $(p,0,0)$ is {\em{positive definite}}
and a matrix with inertia $(0,0,n)$ is {\em{negative definite}}; in either case the matrix is
said to be {\em{definite}}.

\begin{lem*}\nonumber
If matrices $A$ and $M$ with the same dimension are real symmetric and
  invertible with different inertia, then
  $M^{-1}A$ has at least one negative real eigenvalue and thus the
  stationary iteration (\ref{statit}) can not be contractive.
\end{lem*}
\begin{proof}
  For any real parameter $\theta$, the real symmetric matrix
  \begin{equation}T(\theta) = (1-\theta) A + \theta M
\label{T}\end{equation}
    satisfies $T(0) = A, T(1)=M$
  and so it's eigenvalues are real and depend continuously on the entries,
  therefore
  continuously on $\theta$. Since the inertia of $A$ and $M$ are different, at least one
  eigenvalue of $T$ must change sign as $\theta$ varies between $0$ and $1$, hence there must
  exist $\widehat\theta\in (0,1)$ such that $T(\widehat\theta)$ is singular (by the Intermediate
  Value Theorem). Thus
  $$ A + \frac{\widehat\theta}{(1-\widehat\theta)} M$$ is singular, so
  $$\frac{\widehat\theta}{\widehat\theta - 1} <0$$ is an eigenvalue of $M^{-1} A$
    and the result follows.
\end{proof}
  
\begin{rem*} An identical result and proof holds also for Complex Hermitian matrices.
\end{rem*}

\color{black}
\begin{rem*} Having identical inertia is of course not sufficient
for contraction of (\ref{statit}) as trivial examples like $A= diag(1,-1),M=-A$ show.
\end{rem*}
\color{black}

\begin{rem*} This lemma certainly does not rule out the existence of complex eigenvalues for
  $M^{-1} A$, equivalently complex generalised
  eigenvalues of the `pencil' $A-\lambda M$ even in the case of real, symmetric matrices $A,M$.
  Indeed, unless
  at least one of $A,M$ is definite, it is generic that complex conjugate pairs of
  eigenvalues exist. When either $A$ or $M$ are definite, all eigenvalues of
  the pencil $A - \lambda M$ must be real.
\end{rem*}

There is much analysis of matrix splittings when $A$ is positive definite (see for example
the classic book by Richard Varga \cite{Varga}). In this case the above lemma indicates
that $M$ should also be positive definite; this is no surprise! 

However, positive definite preconditioning is commonly applied to indefinite symmetric matrices,
notably
in block diagonal preconditioning of saddlepoint systems
(\cite{BenziGolubLiesen},\cite{partII},\cite{Pestana_W}).
\color{black}
In this situation, the eigenvalues of $M^{-1} A$ must all be real because of the similarity transform
\color{black}
$$(M^{\half})^T (M^{-1} A) M^{-\half} =M^{-\half} A M^{-\half}$$
\color{black}
which is a real symmetric matrix (and so must have all real eigenvalues).
\color{black}
Sylvester's Law of Inertia (see for example \cite{GVL}) ensures in this case that
\color{black}
the number of positive, negative and zero eigenvalues of $M^{-1} A$ and of $A$ are the same
because $M^{-\half} A M^{-\half}$ is congruent to $A$ since $M$ and so $M^{\half}$ are symmetric.
\color{black}
%the inertia
%of the original matrix $A$ is unchanged by preconditioning with positive definite $M$ because
%$A$ is similar to (and so has the same eigenvalues as) the matrix
%which is a congruence transform on $M^{-1} A$.
(Note the unambiguous definition of
$M^{\half} = Q \Lambda^{\half} Q^T$ in terms of the diagonalisation $M=Q \Lambda Q$ when positive
square roots are taken).

For this reason, an iterative method for indefinite systems such as {\minres}\ is required
when definite preconditioning is applied to an indefinite system\color{black}.
\color{black}Convergence will depend
on how small $|p(\lambda)|$ can be for all positive and negative eigenvalues for polynomials,
$p$ that satisfy the consistency condition $p(0)=1$ as the degree of the polynomial is
increased
\color{black}(see \cite[Section 4.1]{ESW} or for a more comprehensive description \cite{Fischer}).
\color{black}
For example for block diagonal preconditioning of the Stokes saddle point system,
see the convergence analysis in \cite{FischerSilvesterWathen}.

At another extreme, {\em{constraint preconditioning}} of saddle point systems
ensures that the inertia of the preconditioner
$$M=\left[\begin{array}{cc} W & B^T \\B & 0\end{array}\right] $$
and the original saddle point matrix 
\begin{equation}A=\left[\begin{array}{cc} H & B^T \\B & 0\end{array}\right]
\label{saddlept}\end{equation}
are identical at least when $W,H$ are positive definite. Indeed, in the paper
\cite{KellerGouldWathen} that introduced constraint preconditioning,
it is proved that all eigenvalues
of the pencil $A-\lambda M$ are real and positive when $W,H$ are positive definite.
Clearly this is consistent with the Lemma above.
\color{black}A consequence is that a specialised and rapidly converging Conjugate Gradient
method can be employed
with constraint preconditioning \cite{Gould_Hribar_Nocedal}. \color{black}
For inexact constraint preconditioning, there are important eigenvalue bounds in
\cite{Bergamaschietal}, \cite{Bergamaschierratum}.
We comment that in the multigrid literature for flow problems, relevant
references are \cite{Braess_Sarazin}, \cite{Zulehner}.

In the common situation where $H\in\R^{m\times m}$ is invertible
in (\ref{saddlept}), the congruence transform of the saddle point matrix
\begin{equation}\nonumber
A=\left[\begin{array}{cc} H & B^T\\ B &0 \end{array}\right]
=
\left[\begin{array}{cc} I & 0\\ B H^{-1} & 0\end{array}\right]
\left[\begin{array}{cc} H & 0\\ 0 & -B H^{-1} B^T \end{array}\right]
\left[\begin{array}{cc} I & H^{-1} B^T\\ 0 & I \end{array}\right]
\end{equation}
is well known. A consequence is that if $H$ is positive definite and
$B\in\R^{n\times m}$ is of full rank, then
the inertia of $A$ must be $(m,0,n)$ by Sylvester's Law of Inertia. For such problems
preservation of this known inertia in a preconditioner is possibly simpler. 

%Less obvious are examples such as
This observation is implicitly employed in Vanka iteration for the indefinite saddle point system
of Stokes flow \cite{Vanka} where an indefinite splitting matrix is used. It may not
be immediately clear whether it shares the same inertia as the original Stokes matrix, however
by careful construction, associating local problems to small numbers of pressure
variables (Lagrange multipliers) often in a `patch', the number of negative eigenvalues
can be guaranteed to equal the number of Lagrange multipliers, $m$, and hence the inertia can be 
preserved \cite{PCPATCH},\cite{FarrellHeMachLachlan}.

Unless the inertia is identical, the above lemma implies that
\color{black}the simple iteration (\ref{statit}) \color{black} could not be contractive.
Further, even standard Chebyshev polynomial acceleration (based on eigenvalues in one interval
to the right of the origin) can not be convergent for any chosen interval if
\color{black} $M^{-1} A$ has \color{black}
negative \color{black} as well as positive \color{black} eigenvalues.
\color{black} This follows because Chebyshev polynomials grow exponentially in magnitude away from
the chosen interval; in turn this implies that $|p(\lambda)|$ must be large (certainly greater than $1$)
for any negative eigenvalue $\lambda$. \color{black}
The most successful Vanka-like iterations must not only satisfy the criterion in the lemma above,
they should cluster eigenvalues in an appropriate way.
%
%The choice of how exactly to construct a Vanka-like
%iteration in such a situation could be guided by the inertia criterion articulated in the
%above lemma;
%
%having identical inertia is of course not sufficient for contraction as trivial
%examples like $A = {\rm{diag}}(1,-1), M=-A$ show.
There are now many examples in the
literature where Vanka-like iteration is used
for smoothing in a multigrid context---for some recent examples see
\cite{Anselmann_Bause}, \cite{Tao_Sifakis}, \cite{He}---even though there remains
limited theoretical literature on the convergence of Vanka-type iterations. Twenty
years ago, Manservisi \cite{Manservisi} declared that there was ``nothing at all on
Vanka-type smoothers'' and Larin and Reusken \cite{Larin_Reusken} say that ``As far as we
know there is no convergence analysis of a multigrid
method with a Vanka smoother applied...to Stokes problems''. Even in 2022, Saberi et al.
\cite{Saberietal}
can only refer to the paper of Manservisi for a proof of convergence of ``some Vanka-type
smoothers for the Stokes...equations'', though Larin and Reusken noted that ``smoothing
properties are not, however, considered in that paper''.
Smoothing does not strictly require contraction, but non-contractive smoothers are not so common.
\color{black}O\color{black}ur theory here only relates to the symmetric Stokes flow situation with block
Jacobi-like smoothing (additive Vanka relaxation)\color{black}, that is a situation where
a symmetric preconditioner is used
for a symmetric matrix problem. Nevertheless, if the iteration (\ref{statit}) is not contractive for
the symmetric Stokes problem, it would seem risky to employ \color{black} a related \color{black}
iteration for more complicated Navier-Stokes problems \color{black}since these problems give rise to
matrices that have additional nonsymmetric terms (see \cite[Chapters 8 \& 9]{ESW}).\color{black}

Indefinite systems necessarily arise for wave problems like the Helmholtz and Maxwell equations.
For the Helmholtz problem for example, inertia will depend on the wave frequency and 
discretisation. Finding a preconditioner with identical inertia in such a situation can be
expected to be challenging (see \cite{Ernst_Gander}).

\section*{Avoidance of crossing}
\label{sec:avoidance}

Because $T(\theta)$ in (\ref{T}) is symmetric for all $\theta$, the concept of {\em {avoidance
of crossing}} or {\em{eigenvalue avoidance}} due to Lax \cite{LaxLA} applies. Though this
may be of less importance than the Lemma above, it allows some slightly more precise 
statements the be made about the number of negative eigenvalues of $M^{-1} A$.

\begin{prop*} If $A$ has inertia $(p,0,n)$ and $M$ has inertia $(p+r,0,n-r)$ where the integer
  $r$ necessarily satisfies $-p\leq r\leq n$, then $M^{-1} A$ has $|r|+2 s$ real
  and negative eigenvalues
  for some $s\in\{0,1,2,\ldots, {\lfloor{\frac{p+n-r}{2}}\rfloor}\}$.
\end{prop*}
\begin{proof}
  Because $T(\theta)$ is real symmetric, eigenvalue avoidance implies that the trajectories
  of the eigenvalues $\lambda(\theta)$ of $T(\theta)$ do not intersect. Thus there must be
  $|r|$ distinct values $\widehat{\theta}_j\in(0,1), j=1,2,\ldots,|r|$ for which
  $T(\widehat{\theta})$ is singular. As in the lemma above, the real value
  $$\frac{\widehat{\theta}_j}{\widehat{\theta}_j - 1} <0$$ must be an eigenvalue of $M^{-1} A$
  for $j=1,2,\ldots,|r|$.

  There may be more values of $\theta\in(0,1)$ where $T(\theta)$
  is singular; if such occurs, then the trajectory $\lambda(\theta)$ must intersect with the
  line $\lambda=0$ an even number of times. The upper bound on $s$ arises simply because there
  can be no more than $m+n$ eigenvalues of $M^{-1} A$ overall.

\end{proof}

\begin{exmpl*}
  The symmetric matrices
$$  A=\left[\begin{array}{ccccc}
      0.33 & -.05 & -.29 & 0.01 & 0.01\\
      -.05 & 0.36 & -.11 & -.22 & -.19\\
      -.29 & -.11 & -.32 & 0.11 & -.01\\
      0.01 & -.22 & 0.11 & 0.49 & -.12\\
      0.01 & -.19 & -.01 & -.12 & 0.18
    \end{array}\right], \quad
  M=\left[\begin{array}{ccccc}
      0.14 & 0.10 & 0.25 & 0.09 & -.28\\
      0.10 & -.07 & 0.02 & 0.08 & -.11\\
      0.25 & 0.02 & 0.49 & -.11 & -.23\\
      0.09 & 0.08 & -.11 & 0.24 & -.34\\
      -.28 & -.11 & -.23 & -.34 & 0.35      
    \end{array}\right]$$
have eigenvalues respectively
$$\lambda_A = -0.4553, -0.0346, 0.3949, 0.4560, 0.6791,$$
$$ \lambda_M = -0.1464, -0.1216, -0.0252, 0.5174, 0.9258,$$
hence in the notation of the above Proposition $p=3,n=2,r=-1$. The eigenvalues of $T(\theta)$ are
plotted as $\theta$ varies from $0$ to $1$; it is seen that $s=1$ in this example.
%  the inertia of $A$ is $(3,0,2)$ and the inertia
%  of $M$ is $(2,0,3)$
%  \begin{figure}[ht]  
\begin{center}
  \epsfig{file=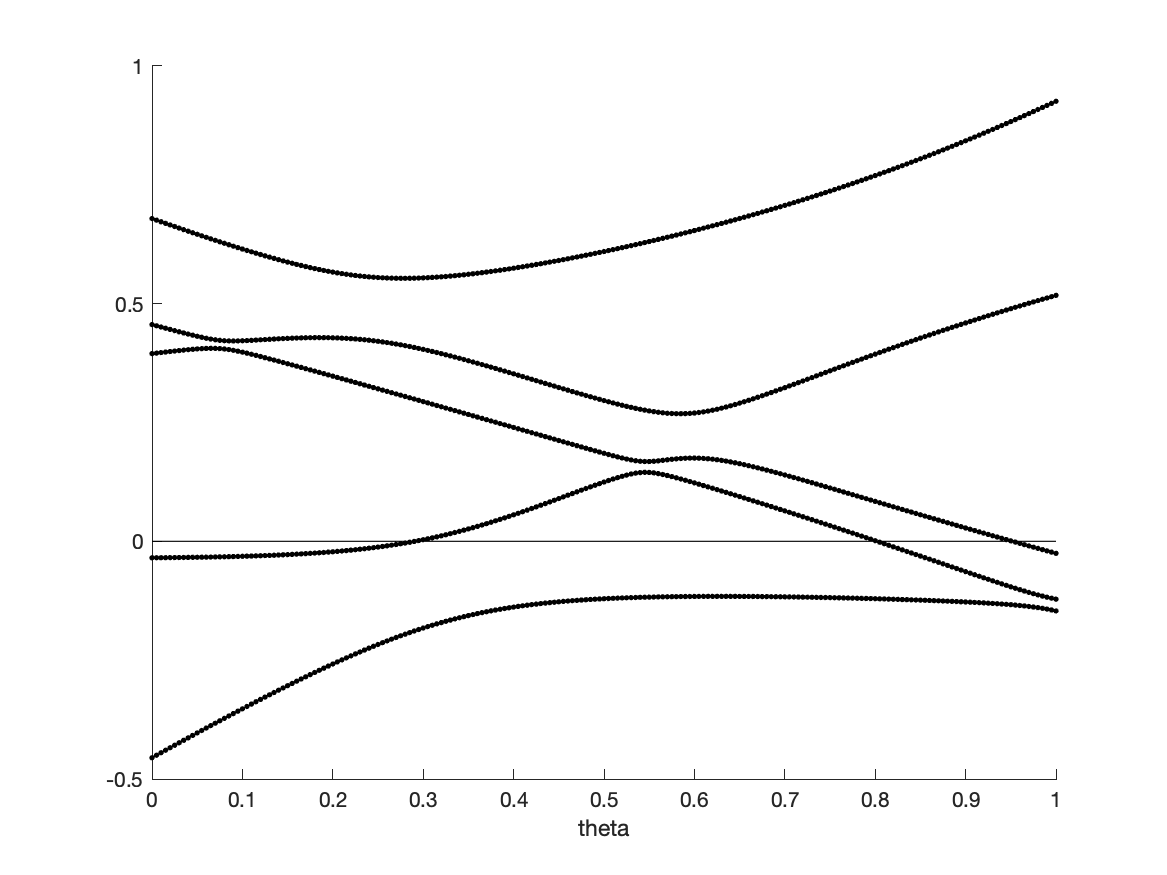, height=3.6in, width=3.6in}
\end{center}
%\caption{}
%\label{tri_grid}
%\end{figure}
%  \end{exmpl}

The eigenvalues of $M^{-1}A$ are:
$$ \lambda_{M^{-1}A} = -2.4405, -0.2468, -0.0506, 1.7245 \pm 0.8315 i
$$
with the three negative real eigenvalues corresponding
respectively to the values
$$\widehat{\theta} = 0.2907, 0.8021, 0.9518 $$ as seen in the plot.
\end{exmpl*}

Rather than $T(\theta)$, by considering the alternative homotopy
$$S(\theta)=(1-\theta)A +\theta(-M)$$
%and noting that $S(0)=A,S(1)=-M$
one can say something about positive real eigenvalues. 

\begin{cor*}
  If $A$ has inertia $(p,0,n)$ and $M$ has inertia $(p+r,0,n-r)$ where the integer
  $r$ necessarily satisfies $-p\leq r\leq n$, then $M^{-1} A$ has $|p+r-n|+2 t$ real
  and positive eigenvalues
  for some
  $t\in\{0,1,2,\ldots, \min{\left({\lfloor{\frac{2p+r}{2}}\rfloor},{\lfloor{\frac{2n-r}{2}}\rfloor}\right)}\}$.
\end{cor*}
\begin{proof}
  Simply note that $S(0) = A$ has inertia $(p,0,n)$ and $S(1) = -M$ has
  inertia $(n-r,0,p+r)$ so that if $p\ne n-r$ then $S(\theta)$ must be singular for
  at least $|p+r-n|$ values $\widehat{\theta}\in(0,1)$ implying $|p+r-n|$ real
  positive eigenvalues $\frac{\widehat{\theta}}{1-\widehat{\theta}}$. Any further
  crossings of the axis must appear in pairs as in the Proposition above.
\end{proof}  

In the example above we have $p=n-r$, (ie. the inertia of $A$ = the inertia of $-M$), hence the
absence of any real positive eigenvalues; here $t=0$.

\section*{Conclusions}

\label{sec:conclusions}

The elementary lemma of this paper
\color{black}
has implications for 
%implies that
preconditioned iterative solution of symmetric linear systems.
Unless the inertia of the coefficient matrix is preserved in a preconditioner/splitting matrix
then simple iteration will generally be non-contractive. Correspondingly when a Krylov
subspace iteration method is used, slower convergence might be expected if the inertia
of the preconditioner does not match the inertia of the coefficient matrix.
%When the inertia does not match, then one generally requires
\color{black}
%almost always requires
%a broadly applicable iterative method such as \gmres.
%Only in situations where the inertia of the preconditioner is identical to the inertia
%of the coefficient matrix might it be possible employ more efficient and specialised methods
%unless the preconditioner is definite when \minres\ can (and should) be used.

Where smoothing for multigrid is based on an indefinite matrix splitting, the smoother
can not be a contractive iteration if the inertia is not exactly preserved.

\section*{Acknowledgments}
I am grateful to Hussam Al Daas for useful conversations about this work
and for the helpful comments of an anonymous referee.
%three anonymous referees, whose comments have clarified and
%greatly improved this short manuscript.

\bibliographystyle{siamplain}
\bibliography{indef_references}

\end{document}